\documentclass[11pt]{amsart}
\usepackage[english]{babel}
\usepackage{amssymb}
\usepackage{fourier}
\usepackage{mathrsfs}
\usepackage{enumerate}
\usepackage{color}
\usepackage{ifpdf}
\usepackage{tikz}
\usepackage{tikz-cd}
\usetikzlibrary{decorations.pathmorphing,arrows}
\usepackage[initials]{amsrefs}

\usepackage{caption}
\usepackage{subcaption}
\usepackage{comment}
\usetikzlibrary{intersections}

\newcommand{\bbN}{{\mathbf N}}

\newcommand{\bbR}{{\mathbf R}}

\newcommand{\bbZ}{{\mathbf Z}}
\newcommand{\bbC}{{\mathbf C}}

\newcommand{\bfH}{{\mathbf H}}

\newcommand{\calC}{\mathcal{C}}

\newcommand{\SL}{\operatorname{SL}}
\newcommand{\PSL}{\operatorname{PSL}}

\newcommand{\tr}{\operatorname{tr}}

\newcommand{\Isom}{\operatorname{Isom}}

\newcommand{\Prob}{\operatorname{Prob}}

\newcommand{\diam}{\operatorname{diam}}

\newcommand{\Homeo}{\operatorname{Homeo}}

\newcommand{\Dg}{\operatorname{\mathscr{D}_{\Gamma}}}
\newcommand{\Teich}{\operatorname{\mathscr{T}}}
\newcommand{\Rmn}{\mathscr{R}}
\newcommand{\QF}{\operatorname{\mathscr{Q}\mathscr{F}}}

\newcommand{\Ax}[1]{\operatorname{Ax}_{{#1}}}

\newcommand{\Lip}{\operatorname{Lip}}

\newcommand{\Bus}{\operatorname{B}}

\newcommand{\overto}[1]{{\buildrel{#1}\over\longrightarrow}}

\newcommand{\setdef}[2]{ \left\{ {#1}\ |\ {#2} \right\} }

\newcommand{\dLip}[2]{\rho_{\operatorname{Lip}}\left({#1},{#2}\right)}

\newtheorem{theorem}{Theorem}[section]

\newtheorem{mthm}{Theorem}

\newtheorem{lemma}[theorem]{Lemma}

\newtheorem{prop}[theorem]{Proposition}

\newtheorem{remark}[theorem]{Remark}

\newtheorem{question}[theorem]{Question}

\numberwithin{equation}{section}

\title{Quasi-Fuchsian vs Negative curvature metrics\\ on surface groups}

\author{Ethan Fricker}
\address{University of Illinois at Chicago}
\email{efrick3@uic.edu}

\author{Alex Furman}
\address{University of Illinois at Chicago}
\email{furman@uic.edu}

\begin{document}
\date{\today}

\begin{abstract}
	We compare two families of left-invariant metrics on a surface group 
	$\Gamma=\pi_1(\Sigma)$ in the context of course-geometry.
	One family comes from Riemannian metrics of negative curvature on the 
	the surface $\Sigma$, and another from quasi-Fuchsian representations
	of $\Gamma$. We show that the Teichmuller space $\Teich(\Sigma)$ 
	is the only common part of these two families, even 
	when viewed from the coarse-geometric perspective.
\end{abstract}

\maketitle

\begin{center}
    \emph{To Benjy Weiss with gratitude and admiration}
\end{center}

\section{Introduction and Statement of the main result} 

\subsection{Introduction and Background}
Let $\Sigma$ be a closed surface of genus at least two, 
and $\Gamma=\pi_1(\Sigma)$ its fundamental group. 
The Teichmuller space $\Teich(\Sigma)$ has several equivalent descriptions:
as the moduli space of (i) complex structures, or (ii) conformal structures, 
or (iii) Riemannian structures of constant curvature $-1$  on $\Sigma$,
or as (iv) the space of discrete cocompact representations 
$\Gamma\to \PSL_2(\bbR)$, up to conjugation.
The latter two points of view can be extended as follows:
\begin{itemize}
    \item $\Rmn(\Sigma)$ -- the space of all Riemannian structures 
    of possibly variable negative curvature, up to isotopy and scaling.
    \item $\QF(\Sigma)$ -- the space of all convex cocompact representations 
    \[
        \Gamma=\pi_1(\Sigma)\overto{} \PSL_2(\bbC)\cong\Isom^+(\bfH^3),
    \]
    up to conjugation. 
\end{itemize}
Both $\Rmn(\Sigma)$ and $\QF(\Sigma)$ arise from convex cocompact isometric
$\Gamma$-actions on CAT(-1) spaces: 
the $\Gamma$-action by deck transformations on the
universal cover $(\tilde{\Sigma},d_{\tilde{g}})$ in the Riemannian case, 
and the $\Gamma$-action on $\bfH^3$ in the quasi-Fuchsian case.

\medskip

We can put these notions into an even broader context by looking at the space 
$\Dg$ of equivalence classes $[d]$ of
left-invariant metrics $d$ on $\Gamma$ obtained from restricting 
the metric of the underlying 
Gromov-hyperbolic space to a $\Gamma$-orbit.
Here two metrics $d, d'$ on $\Gamma$ are equivalent if they are bounded distance 
from each other after scaling:
\[
    d\sim d'\qquad\textrm{if}\qquad \exists k, A:\quad |d'(\gamma_1,\gamma_2)-k\cdot d(\gamma_1,\gamma_2)|\le A.
\] 
This perspective, introduced by the second author in \cite{Fur} 
(see also more recent treatment in Bader--Furman \cite{BF}),  
allows to observe possible "geometries" of $\Sigma$ from the "outside" by 
studying the corresponding classes $[d]\in\Dg$ of metrics $d$ on $\Gamma$.
The space $\Dg$ can be defined for a general non-elementary Gromov hyperbolic group $\Gamma$,
and $\Dg$ contains classes of metrics on $\Gamma$ from various sources, such as 
word metrics on $\Gamma$, Green metrics associated with symmetric generating 
random walks on $\Gamma$ (see Blach{\`e}re--Ha{\"{\i}}ssinsky--Mathieu \cites{BHM1, BHM2}), 
Anosov representations of $\Gamma$ in higher rank simple Lie groups 
(see Dey--Kapovich \cite{DK}), etc.

To avoid ambiguity in scaling we can normalize metrics $d$ by the growth
\[
    h_d=\lim_{R\to\infty} \frac{1}{R}\log\#\setdef{\gamma\in\Gamma}{d(\gamma,e)<R},     
\]
replacing $d$ by $\hat{d}=h_d\cdot d$, so that $h_{\hat{d}}=1$.
For $\delta\in \Dg$ we can define:
\begin{itemize}
    \item 
    Marked Length Spectrum $\ell_\delta:\Gamma\to\bbR_+$ given by the limit
    \[
        \ell_\delta\left(\gamma\right)=\lim_{n\to\infty} \frac{\hat{d}(\gamma^n,e)}{n}
    \]
    where $\delta=[d]$ and $\hat{d}=h_d\cdot d$. 
    Note that $\ell_\delta$ is constant on conjugacy classes,
    so we can write it as $\ell_\delta:\calC_\Gamma\to\bbR_+$.
    \item 
    Patterson-Sullivan-like $\Gamma$-invariant measure class 
    $[\nu_\delta^{\rm PS}]$ on $\partial\Gamma$ (see Coorneart \cite{Coor},
    and \cites{Fur, BF}).
    \item 
    Bowen--Margulis--Sullivan-like $\Gamma$-invariant Radon measure $m_\delta^{\rm BMS}$
    on the space $\partial^{(2)}\Gamma$ of distinct pairs $(\xi,\eta)$ of points on $\partial\Gamma$ (see \cites{Fur, BF}).
\end{itemize}
In  \cite{Fur} (see also Bader--Furman \cite{BF}), it was shown that each $\delta\in\Dg$
is determined by each of these objects. 
Furthermore, extending a prior work of Bader--Muchnik \cite{BM}, Garncarek \cite{Garn}
showed that for each $\delta\in\Dg$ the quasi-regular unitary $\Gamma$-representation 
\[
    \pi_\delta:\Gamma\overto{}U(\partial\Gamma,\nu_\delta^{\rm PS})
\]
is irreducible, and that the map $\Dg\overto{}\hat\Gamma$, $\delta\mapsto \pi_\delta$,
is also injective.
Thus $\Dg$ can be embedded into any obe of the the following spaces:
\[
    \bbR_+^{\calC_\Gamma}, \qquad\Prob(\partial\Gamma),\qquad
    {\rm Meas}_\Gamma\left(\partial^{(2)}\Gamma\right),\qquad \hat{\Gamma}.
\]
The space $\Dg$ is also equipped with a natural metric:
given two classes $\delta=[d]$, $\delta'=[d']$ in $\Dg$ we can define the
(log) Lipschitz distance by
\[
    \dLip{\delta}{\delta'}:=\log\left( \inf 
    \setdef{\frac{K}{k}}{\exists A,\ k\cdot d-A\le d'\le K\cdot d+A}\right).
\]
It is clear from the definition that $\dLip{-}{-}$ is symmetric 
and satisfies the triangle inequality.
One can see that for any $a,b\in\Gamma\setminus\{e\}$ one has 
\[
    \left|\log\left(\frac{\ell_\delta(a)}{\ell_\delta(b)}:
    \frac{\ell_{\delta'}(a)}{\ell_{\delta'}(b)}\right)\right|
    \le \dLip{\delta}{\delta'}.
\]
This shows that $\dLip{\delta}{\delta'}=0$ implies $\ell_\delta=\ell_{\delta'}$,
which occurs only when $\delta=\delta'$. So $\dLip{}{}$ is indeed a metric on $\Dg$
(see also a recent work of Cantrell--Tanaka \cite{Cantrell+Tanaka} for a more
detailed picture).

\subsection{Riemannian and Quasi-Fuchsian structures on Surfaces}
In this paper we focus on surface group $\Gamma=\pi_1(\Sigma)$ and two specific
sources for $\delta\in\Dg$: namely $\Rmn(\Sigma)$ and $\QF(\Sigma)$.

For the case of negatively curved Riemannian metric $g$ on $\Sigma$, 
fix $x\in \tilde{\Sigma}$ and consider the metric on $\Gamma$
\[
    d_{g,x}(\gamma_1,\gamma_2):=d_{\tilde{g}}(\gamma_1x,\gamma_2x).
\]
Since $|d_{g,x}-d_{g,x'}|\le \diam(\Sigma,g)$ the class $[d_{g,x}]$ 
does not depend on the choice of $x\in \tilde{\Sigma}$, 
and we can denote this class by $\delta_g=[d_{g,x}]$.
Note that $h_{d_{g,x}}$ is the topological entropy of the geodesic flow
on the unit tangent bundle $T^1\Sigma$ to $\Sigma$, and we assume that all $g\in\Rmn(\Sigma)$
are normalized so that $h_{d_{g,x}}=1$.
We have a map
\begin{equation}\label{e:Rmn2Dg}
    i:\Rmn(\Sigma)\ \overto{}\ \Dg.
\end{equation}
The Marked Length Spectrum Rigidity Conjecture, that for surfaces
was proved by Otal \cite{Otal} and Croke \cite{Croke}, 
asserts that a Riemannian structure $g$ of variable negative curvature on a surface $\Sigma$ 
is uniquely determined by the function $\ell_g:\calC_\Gamma\to\bbR$.
As a consequence, we obtain:
\begin{prop}\label{P:MLSR}
    The map $\Rmn(\Sigma)\overto{} \Dg$, $i:g\mapsto \delta_g$, is injective.
\end{prop}

Our second source of examples, are quasi-Fuchsian representations.
For $q\in\QF(\Sigma)$ choose a representation 
$\pi:\Gamma\to \Isom^+(\bfH^3)\cong \PSL_2(\bbC)$ 
in this class and a point $y\in\bfH^3$ and consider 
the metric on $\Gamma$:
\[
    d_{\pi,y}(\gamma_1,\gamma_2):=d_{\bfH^3}\left(\pi(\gamma_1).y,\pi(\gamma_2).y\right).
\]
The class $[d_{\pi,y}]$ does not depend on the choice of $y\in\bfH^3$
and remains unchanged if $\pi$ is replaced by
a conjugate $\gamma\mapsto g \pi(\gamma)g^{-1}$;
thus we write $\delta_q$ for $[d_{\pi,y}]$.
This gives a well defined map 
\begin{equation}\label{e:QF2Dg}
    j:\QF(\Sigma)\ \overto{}\ \Dg.
\end{equation}
One can deduce from a work of Burger \cite{Burger} (or Dal'bo--Kim \cite{Dalbo+Kim}), 
the following.
\begin{prop}\label{P:QFR}
    The map $\QF(\Sigma)\overto{} \Dg$, $j:q\mapsto \delta_{q}$, is injective.
\end{prop}
Hence one might view each of $\Rmn(\Sigma)$ and $\QF(\Sigma)$ as being 
embedded in $\Dg$. 

\begin{remark}\label{R:connected-images}
    We note in passing that the uniformization theorem allows us to view
$\Rmn(\Sigma)$ as a bundle over $\Teich(\Sigma)$ with fibers that can be identified
with the positive cone $C^\infty_+(\Sigma)/\bbR_+$; in particular $\Rmn(\Sigma)$ is connected.
One can show that the map (\ref{e:Rmn2Dg}) is continuous, and so the image $i(\Rmn(\Sigma))$ 
in $\Dg$ is connected.  

Ahlfors and Bers showed that $\QF(\Sigma)$ can be identified with $\Teich(\Sigma)\times\Teich(\Sigma)$, and is in particular connected. 
The map (\ref{e:QF2Dg}) can be shown to be continuous;
hence the image $j(\QF(\Sigma))$ is a connected subset of $\Dg$. 
\end{remark}

It is natural to wonder whether the intersection
\[
    i(\Rmn(\Sigma))\cap j(\QF(\Sigma))\subset \Dg
\]
contains anything except for the image of $\Teich(\Sigma)$.
In other words, is it true that given a quasi-Fuchsian representation 
$\pi:\Gamma\overto{}\PSL_2(\bbC)$
and a negatively curved metric $g$ on the surface $\Sigma$, there exist constants $k,A$ and points $x\in\tilde{\Sigma}$, $y\in\bfH^3$, so that 
\[
    k \cdot d_{\tilde{g}}(\gamma.x,x)-A\le d_{\bfH^3}(\pi(\gamma).y,y)\le 
    k\cdot d_{\tilde{g}}(\gamma.x,x)+A\qquad(\gamma\in\Gamma)
\]
only if $g$ has constant curvature, $\pi$ is conjugate into $\PSL_2(\bbR)$,
and $(\Sigma,g)$ and $\pi$  represent the same point in $\Teich(\Sigma)$?

Our main result answers this affirmatively.
\begin{mthm}\label{T:main}
    The images of $\Rmn(\Sigma)$ and $\QF(\Sigma)$ in $\Dg$ have only $\Teich(\Sigma)$ 
    in common. Moreover, for any $q\in \QF(\Sigma)\setminus \Teich(\Sigma)$ there is
    $\alpha_q>0$ so that
    \[
        \dLip{\delta_q}{\delta_g}\ge \alpha_q>0
    \]
    for all $g\in\Rmn(\Sigma)$.
\end{mthm}

The following natural question remains open.
\begin{question}
    Is it true that for $g\in\Rmn(\Sigma)\setminus\Teich(\Sigma)$ there exists $\beta_g>0$
    so that
        \[
        \dLip{\delta_q}{\delta_g}\ge \beta_g>0
    \]
    for all  $q\in\QF(\Sigma)$?
\end{question}

\subsection*{Acknowledgement}
This note is dedicated to Benjy Weiss on the occasion of his $80$th birthday.
His profound contributions to Ergodic Theory and Dynamics, 
breadth of his interests and originality of his ideas 
are an inspiration to us and many, many 
others.

We would like to thank Dick Canary for the useful discussion about quasi-fuchsian representations.

A.F. acknowledges the support of the NSF grant DMS-2005493 and BSF grant-2018258.

\section{Length inequalities for negatively curved surfaces}
\label{sec:neg-curv} 

Consider the topological picture first.
Let $\Sigma$ be a closed surface of genus at least two, $\Gamma=\pi_1(\Sigma)$ the corresponding surface group, that acts on the universal cover $\tilde{\Sigma}$ by deck transformations.
This action extends to the action of $\Gamma$ on the boundary circle $\partial\tilde{\Sigma}$,
which is also the Gromov boundary $\partial\Gamma$ of $\Gamma$.
Every $\gamma\ne 1$ in $\Gamma$ has two fixed points on the
topological circle $\partial\tilde{\Sigma}$:
a repelling point $\gamma^-$ and an attracting point $\gamma^+$.
We shall consider a pair $a,b\in\Gamma$ where $a^-$, $a^+$, $b^-$, $b^+$
are four distinct points on the circle.

\medskip

Let $A=(\alpha_1,\alpha_2)$ and $B=(\beta_1,\beta_2)$ be two ordered pairs 
on a circle $C$, where all four points are distinct. 
The action of $\Homeo(C)$ on such pairs has $3$ orbits corresponding to $3$ 
possible relative positions of the two pairs $A$, $B$:
\begin{itemize}
    \item The pairs are \textbf{linked}, meaning that 
    $\beta_1$ and $\beta_2$ lie in distinct arcs defined by 
    $\{\alpha_1,\alpha_2\}$ -- connected components 
    of $C\setminus\{\alpha_1,\alpha_2\}$.
    The relation of being linked is symmetric: $A$ is linked with $B$ iff $B$ is linked with $A$. The order within the pairs $A=(\alpha_1,\alpha_2)$ and $B=(\beta_1,\beta_2)$ does
    not change the status of being linked.
    We say that disjoint pairs $A$ and $B$ are \textbf{unlinked} if they are not linked.
    \item The pairs $A$ and $B$ are \textbf{unlinked and aligned}, 
    if in the arc $\widearc{\alpha_1,\alpha_2}$ determined by $\{\alpha_1,\alpha_2\}$ 
    on $C$ containing $\beta_1$ and $\beta_2$ one has linear order $\alpha_1<\beta_1<\beta_2<\alpha_2$. We note that $A$ is unlinked and aligned with $B$ iff $B$ is unlinked and aligned with $A$. In this case flipping the order in both pairs $A$ and $B$ simultaneously does not change the status of being aligned.
    \item The pairs $A$ and $B$ are \textbf{unlinked and misaligned}, if in the arc $\widearc{\alpha_1,\alpha_2}$ determined by $\{\alpha_1,\alpha_2\}$ on $C$ containing $\beta_1$ and $\beta_2$ one has linear order $\alpha_1<\beta_2<\beta_1<\alpha_2$. 
    We note that $A$ is unlinked and misaligned with $B$ iff $B$ is unlinked and misaligned with $A$. In this case flipping the order in both of $A$ and $B$ simultaneously does not change the status of being misaligned.
    Yet flipping the order in either $A$ or $B$ makes the pair unlinked and aligned.
\end{itemize}

\medskip

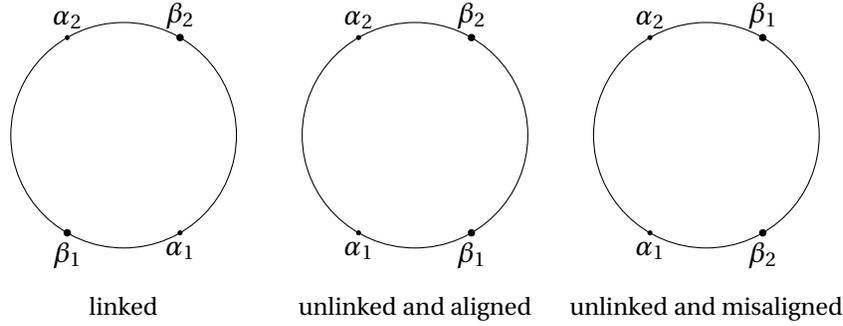
\begin{figure}[h]
\centering
\begin{subfigure}[b]{.3\textwidth}
\centering
\begin{tikzpicture}[scale=1.5]
\draw (0,0) circle(1cm);
\filldraw (-60:1cm) circle(.5pt) node[below]{$\alpha_1$}  (120:1cm) circle(.5pt) node[above]{$\alpha_2$};
\filldraw (-120:1cm) circle(.8pt) node[below]{$\beta_1$} (60:1cm) circle(.8pt) node[above]{$\beta_2$};
\end{tikzpicture}
\caption*{linked}
\end{subfigure}
\begin{subfigure}[b]{.3\textwidth}
\centering
\begin{tikzpicture}[scale=1.5]
\draw (0,0) circle(1cm);
\filldraw (-60:1cm) circle(.8pt) node[below]{$\beta_1$} (60:1cm) circle(.8pt) node[above]{$\beta_2$};
\filldraw (-120:1cm) circle(.5pt) node[below]{$\alpha_1$} (120:1cm) circle(.5pt) node[above]{$\alpha_2$};
\end{tikzpicture}
\caption*{unlinked and aligned}
\end{subfigure}
\begin{subfigure}[b]{.3\textwidth}
\centering
\begin{tikzpicture}[scale=1.5]
\draw (0,0) circle(1cm);
\filldraw (-60:1cm) circle(.8pt) node[below]{$\beta_2$} (60:1cm) circle(.8pt) node[above]{$\beta_1$};
\filldraw (-120:1cm) circle(.5pt) node[below]{$\alpha_1$} (120:1cm) circle(.5pt) node[above]{$\alpha_2$};
\end{tikzpicture}
\caption*{unlinked and misaligned}
\end{subfigure}
\end{figure}

\medskip

Let us now choose a negatively curved Riemannian metric $g$ on $\Sigma$,
and let $\tilde{g}$ be its lift to $\tilde{\Sigma}$.
Denote by $d_{\tilde{g}}$ the corresponding distance on $\tilde{\Sigma}$,
and by $\ell_g:\Gamma\to [0,\infty)$ the associated \textit{stable length}
\[
    \ell_g(\gamma):=\lim \frac{1}{n} d_{\tilde{g}}(\gamma^n.p,p)
\]
where $p\in\tilde{\Sigma}$ is arbitrary. 

\begin{theorem}\label{T:NC-ineq}
    Let $a,b\in\Gamma$ be non-trivial elements with distinct fixed points 
    $a-_,a^+,b^-, b^+$ on the boundary circle $\partial\Gamma$.
    Then 
    \begin{enumerate}
    \item If $(a^-,a^+)$ and $(b^-,b^+)$ are linked, then 
        \[
            \ell_g(ab)<\ell_g(a)+\ell_g(b).
        \]
     \item If $(a^-,a^+)$ and $(b^-,b^+)$ are unlinked and aligned, then 
        \[
            \ell_g(ab)>\ell_g(a)+\ell_g(b).
        \]
     \item If $(a^-,a^+)$ and $(b^-,b^+)$ are unlinked and misaligned, then 
        \[
            \ell_g(a^{-1}b)>\ell_g(a)+\ell_g(b).
        \]
    \end{enumerate}
\end{theorem}

\begin{proof}
    First recall that in the case of negatively curved manifolds, such as $(\Sigma,g)$ the stable length $\ell_g(\gamma)$ 
can also be defined as the \emph{minimal translation length}
\[
    \ell_g(\gamma)=\inf_{p\in\tilde\Sigma} d_{\tilde{g}}(\gamma.p,p).
\]
Moreover, when $\ell_g(\gamma)>0$, which is the case of any non-trivial $\gamma\ne 1$, the $\inf$ is attained
and the set 
\[
    \Ax{\gamma}:=\setdef{p\in \tilde\Sigma}{d_{\tilde{g}}(\gamma.p,p)=\ell_g(\gamma)}
\]
is the geodesic line $(\gamma^-,\gamma^+)$ in $\tilde\Sigma$. 
It is called the \textit{axis} of $\gamma$.

\medskip

Elementary topology of the disc $\tilde\Sigma$ implies that when $(a^-,a^+)$ and $(b^-,b^+)$
are linked, the axes $\Ax{a}$ and $\Ax{b}$ must intersect in $\tilde\Sigma$. 
Due to negative curvature the intersection is a singleton: $\Ax{a}\cap\Ax{b}=\{p\}$. 
Since $p\in\Ax{b}$, we have $x=b^{-1}.p\in \Ax{b}$. 
Similarly, we have $p$ and $y=a.p$ are in $\Ax{a}$ as well. 
To prove part $(1)$ we use the triangle inequality to obtain for $x=b^{-1}.p$: 
\[
    \begin{split}
      \ell_g(ab) \leq d_{\tilde g}(x,ab.x) &< d_{\tilde g}(x,b.x) + d_{\tilde g}(b.x,ab.x)\\
        &=d_{\tilde g}(b^{-1}.p,p)+d_{\tilde g}(p,a.p)
        = \ell_g(b) + \ell_g(a)
    \end{split}
\]
We observe that the second inequality is strict and will sharpen it 
in the proof of Theorem~\ref{T:main} below.  

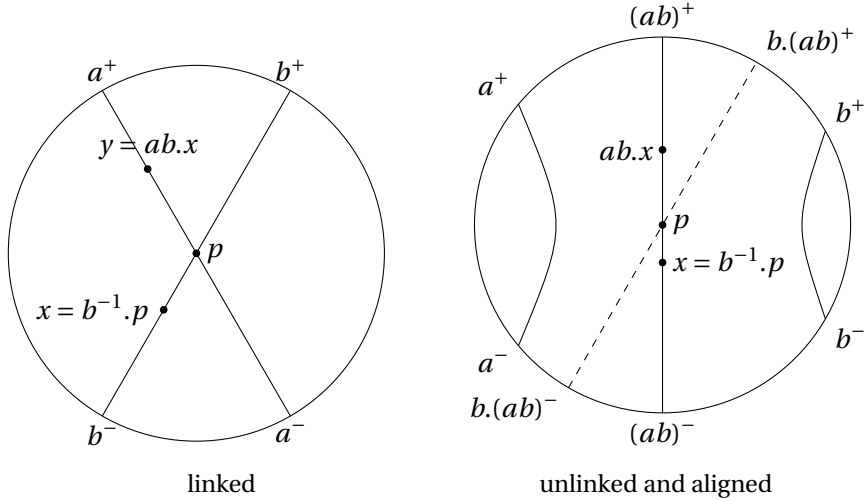
\begin{figure}[h]
\centering
\begin{subfigure}[b]{.45\textwidth}
\begin{tikzpicture}[scale=2.5]
\draw (0,0) circle(1cm);
\draw (-60:1cm) node[below]{$a^-$} -- (120:1cm) node[above]{$a^+$};
\draw (-120:1cm) node[below]{$b^-$} -- (60:1cm) node[above]{$b^+$};
\filldraw (0,0) circle(.5pt) node[right]{$p$};
\filldraw (intersection of -120:1cm--60:1cm and -100:1cm--100:1cm) circle(.5pt) node[left]{$x=b^{-1}.p$ \hspace{1em}};
\filldraw (intersection of -60:1cm--120:1cm and -105:1cm--105:1cm) circle(.5pt) node[above]{$y=ab.x$};
\end{tikzpicture}
\caption*{linked}
\end{subfigure}
\begin{subfigure}[b]{.45\textwidth}
\begin{tikzpicture}[scale=2.5]
\draw (0,0) circle(1cm);
\draw (-30:1cm) node[below right]{$b^-$} .. controls (.7,0) .. (30:1cm) node[above right]{$b^+$};
\draw (-140:1cm) node[below left]{$a^-$} .. controls(-.5,0) .. (140:1cm)node[above left]{$a^+$};
\draw (0,-1) node[below]{$(ab)^-$} -- (0,1) node[above]{$(ab)^+$};
\draw[dashed] (-120:1cm) node[below left]{\hspace{.5em} $b.(ab)^-$} -- (60:1cm) node[above right]{$b.(ab)^+$};
\filldraw (intersection of -120:1cm--60:1cm and 0,-1--0,1) circle(.5pt) node[right]{$p$};
\filldraw (0,-.2) circle(.5pt) node[right]{$x=b^{-1}.p$};
\filldraw (0,.4) circle(.5pt) node[left]{$ab.x$};
\end{tikzpicture}
\caption*{unlinked and aligned}
\end{subfigure}
\end{figure}

In the case where the pairs $(a^-,a^+)$ and $(b^-,b^+)$ are unlinked and aligned, 
we remind ourselves of the definition, that $a^-,a^+$ define an arc $\widearc{a^-a^+}$ 
on the boundary circle containing both $b^-$ and $b^+$, which can be equipped
with a linear order (anti-clockwise in the figure) so that 
\[
    a^- < b^- < b^+ < a^+.
\]
The action of $b$ on the arc/interval from $b^+$ to $a^+$ is decreasing towards the fixed point $b^+$, while the action of $a$ is increasing towards $a^+$.
Thus $ab$ maps this interval into itself, and therefore the attracting point $(ab)^+$ 
satisfies $b^+<(ab)^+<a^+$. Moreover, we have
\[
    b^+=b.b^+<b.(ab)^+<(ab)^+.
\]
Since the repelling fixed point of an element is the attracting fixed point 
of it's inverse, the same argument gives 
$a^-<(ab)^-<b^-$. 
We claim that $a^-<b.(ab)^-<(ab)^-$.
Indeed, in the linear order on the arc $\widearc{b^+b^-}$ that contains $a^\pm$ so that
$b^+<a^+,a^-<b^-$ the map $b$ is decreasing, and thus $\xi=b.(ab)^-<(ab)^-$.
Since $a.\xi=(ab).(ab)^-=(ab)^->\xi$ we deduce that $a^-<\xi<(ab)^-$.
Hence
\[
    a^-<b.(ab)^-<(ab)^-.
\]
We conclude that pair $\left((ab)^-,(ab)^+\right)$ is linked with its image under $b$.
Denote by $p$ the intersection of $\Ax{ab}$ and $b.\Ax{ab}$ in $\tilde\Sigma$, and
let $x=b^{-1}.p$. Since $p\in b.\Ax{ab}$ we have $x\in \Ax{ab}$ and $ab.x \in\Ax{ab}$ as well. 
Thus the points $x$, $p=b.x$, $ab.x=a.p$ lie on the geodesic line $\Ax{ab}$, 
and in fact in this linear order.
This can be seen by inspecting the projections of these points to $\Ax{a}$ and $\Ax{b}$,
making use of the assumption that the pairs are aligned.
Hence
\[
    \ell_g(ab)=d_{\tilde g}(x,ab.x)=d_{\tilde g}(x,b.x)+d_{\tilde g}(p,a.p)
    >\ell_g(b)+\ell_g(a).
\]
The strict inequality here occurs because $p\not\in\Ax{a}$ and $x\not\in\Ax{b}$.
This proves statement (2).

Statement (3) follows from (2) by replacing $a$ by $a^{-1}$.
This completes the proof of Theorem~\ref{T:NC-ineq}.
\end{proof}

\section{Spiraling of the boundary of a quasi-fuchsian embedding}

Let $\Gamma=\pi_1(\Sigma)$ be a surface group, and $q\in \QF(\Sigma)$ 
be defined by a representation $\pi:\Gamma\to\PSL_2(\bbC)$.
For $\gamma\in\Gamma$ the element $g=\pi(\gamma)\in\PSL_2(\bbC)$ 
has two preimages $\pm\hat{g}$ in $\SL_2(\bbC)$.
Since the traces $\pm\tr(\hat{g})$ are invariant under conjugation,
we can denote them by $\pm\tr_q(\gamma)$.
The following is a particular case of a lemma of Vinberg \cite{Vinberg}
(see \cite{MR}*{Corollary 3.2.5}) 
\begin{lemma}\label{L:complex}
    Let $\Gamma=\pi_1(\Sigma)$ a surface group, and 
    $q\in\QF(\Sigma)\setminus \Teich(\Sigma)$.
    Then there exists $\gamma\in\Gamma$ with
    $\ \pm\tr_q(\gamma)\in\bbC\setminus \bbR$.
\end{lemma}

Let $\pi:\Gamma\to \PSL_2(\bbC)$ be a quasi-Fuchsian representation. 
There exists a $\Gamma$-equivariant
continuous map  
\[
    \phi:\partial\Gamma \overto{} \mathbb{P}_\bbC^1,\qquad \phi\circ \gamma=\pi(\gamma)\circ \phi
\]
that is a homeomorphism between the topological circle $\partial\Gamma$ 
and the Jordan curve on the sphere $\mathbb{P}_\bbC^1$
formed by the limit set $L_{\pi(\Gamma)}$ of $\pi(\Gamma)$. 

\begin{prop}\label{P:cyclic-order}
    Let $q\in\QF(\Sigma)\setminus \Teich(\Sigma)$ be given by
    a quasi-Fuchsian representation $\pi:\Gamma\overto{}\PSL_2(\bbC)$.
    Then there exists an isometrically embedded hyperbolic plane 
    $\bfH^2\subset \bfH^3$ and a sequence $\xi_1,\xi_2,\dots\to \xi_* \in \partial \Gamma$ whose cyclic order with respect to the circle $\partial\Gamma$ is 
    \[
        \xi_1,\ \xi_2,\ \xi_3,\ \xi_4,\ \dots,\ \xi_*
    \]
    and whose images $\phi(\xi_n)\in \mathbb{P}_\bbC^1$ lie on the boundary circle $\partial\bfH^2$ in the following cyclic order:
    \[
        \phi(\xi_1),\ \phi(\xi_3),\ \phi(\xi_5),\ \dots ,\ \phi(\xi_*),\ 
        \dots\ \phi(\xi_6),\  \phi(\xi_4),\ \phi(\xi_2).
    \]
    In particular, we have 
    \begin{itemize}
        \item 
        $(\xi_1, \xi_4)$ and $(\xi_2,\xi_3)$ are unlinked and aligned in $\partial\Gamma$,
        while $(\phi(\xi_1), \phi(\xi_4))$ and $(\phi(\xi_2),\phi(\xi_3))$ are linked in $\partial \bfH^2$.
        \item 
        $(\xi_1, \xi_3)$ and $(\xi_2,\xi_4)$ are linked in $\partial\Gamma$,
        while $(\phi(\xi_1), \phi(\xi_3))$ and $(\phi(\xi_2),\phi(\xi_4))$ are unlinked and aligned in $\partial \bfH^2$.
    \end{itemize}
\end{prop}
\begin{proof}
    Fix an element $\gamma\in\Gamma$ with $\pm\tr_q(\gamma)\in\bbC\setminus \bbR$ 
    as in Lemma~\ref{L:complex}.
    Note that $\gamma$ must be hyperbolic, and 
    denote by $\xi_*$ the attracting point $\gamma^+\in\partial\Gamma$.
    At the same time $\pi(\gamma)\in\PSL_2(\bbC)$ is loxodromic with
    an attracting point $\phi(\gamma^+)$.
    Identifying $\mathbb{P}_\bbC^1$ with $\bbC\cup\{\infty\}$ and replacing 
    $\pi:\Gamma\overto{}\PSL_2(\bbC)$ by an appropriate conjugate we may assume
    $\phi(\gamma^+)=\infty$ and $\phi(\gamma^-)=0$.
    Then the action of $\pi(\gamma)$ on $\bbC$ is given by the linear map
    \[
        z\mapsto (\lambda e^{2\pi i\theta})\cdot z
        \qquad\textrm{with}\qquad
        \lambda>1,\qquad \theta\in\bbR\setminus \bbZ.
    \]
    Identify $\partial\Gamma\setminus \{\gamma^+\}$ with $\bbR$ so
    that $\gamma^-$ corresponds to $0\in\bbR$.
    With a slight abuse of notation we write $\gamma$ and $\phi$ 
    for the corresponding homeomorphism of $\bbR$, 
    and an equivariant injective continuous map $\bbR\overto{}\bbC$.
    Note that $\gamma(0)=0$, $\gamma$ is strictly increasing on $[0,\infty)$ 
    (and strictly decreasing on $(-\infty,0]$), while $\phi$ satisfies
    \begin{equation}\label{e:phi-lambda-theta}
        \phi(\gamma (t))=\lambda e^{2\pi i\theta}\cdot \phi(t)
    \end{equation}
    and $|\phi(t)|\to\infty$ as $|t|\to\infty$. 
    Since $\phi(t)\ne 0$ for all $t\in(0,\infty)$ there exist
    continuous functions $r:(0,\infty)\to (0,\infty)$ and $s:(0,\infty)\to\bbR$ 
    so that 
    \[
        \phi(t)=r(t)\cdot e^{2\pi i \cdot s(t)} \qquad(t>0).
    \]
    Thus (\ref{e:phi-lambda-theta}) implies that 
    \[
        r(\gamma^n(t))=\lambda^n\cdot r(t),
        \qquad
        s(\gamma^n(t))=s(t)+n\Theta
    \]
    where $\Theta\in \theta+\bbZ$.
    Note that the assumption that 
    $q\in\QF(\Sigma)\setminus\Teich(\Sigma)$ gives $\theta\not\in\bbZ$ (Lemma~\ref{L:complex}),
    implying $\Theta\ne 0$. 
    
    Fix $t_0>0$ and use points $t_n=\gamma^n(t_0)$, $n\in\bbZ$, 
    to partition the ray $(0,\infty)$.
    Let 
    \[
        R_0=\max\setdef{r(t)}{ 0\le t\le t_0},
        \qquad
        r_0=\min\setdef{r(t)}{ t_0\le t<\infty}.
    \]
    Then $|\phi(t)|=r(t)\le \lambda^n\cdot R_0$ for all $t\in [0,t_n]$,
    and $|\phi(t)|=r(t)\ge \lambda^n\cdot r_0$ for all $t\ge t_n$.
    We can now choose integers $n(1)<m(1)<n(2)<m(2)<\dots$
    so that 
    \[
        (m(k)-n(k))\cdot |\Theta|>1,\qquad \lambda^{n(k+1)-m(k)}>R_0/r_0
    \]
    for all $k\in\bbN$. 
    The first condition guarantees that  $s(t_{m(k)})>s(t_{n(k)})+1$ and therefore
    there exist 
    \[
        \xi_k\in [t_{n(k)},t_{m(k)}]
        \qquad\textrm{with}\qquad
        e^{2\pi i \cdot s(\xi_k)}=(-1)^k.
    \]
    Thus $\phi(\xi_k)=(-1)^k r(\xi_k)$ lie on the real line $\bbR\subset\bbC$ on both sides of $0\in\bbR$ in alternating order.
    Since $\xi_k\le t_{m(k)}<t_{n(k+1)}\le \xi_{k+1}$ we also have
    \[
        |\phi(\xi_k)|=r(\xi_k)\le \lambda^{m(k)}\cdot R_0 
        < \lambda^{n(k+1)}\cdot r_0\le r(\xi_{k+1})=|\phi(\xi_{k+1})|.
    \]
    Thus the sequence $\{|\phi(\xi_k)|\}$ is monotonically increasing.
    In particular, we have 
    \[
        \dots \phi(\xi_5)<\phi(\xi_3)<\phi(\xi_1)<0< \phi(\xi_2)<\phi(\xi_4)<\phi(\xi_6)<\dots 
    \]
    on $\bbR\subset \bbC$.  Recalling that $\phi(\xi_*)=\infty$ we get the
    required cyclic  order.
\end{proof}

\section{Proof of Theorem~\ref{T:main}}

Let us first recall two general well known facts, 
one related to CAT(-1) spaces $(X,d_X)$, and 
another to Gromov hyperbolic groups $\Gamma$ acting on their boundary $\partial\Gamma$. 
We will apply them to $X=\bfH^3$ and to the surface group $\Gamma=\pi_1(\Sigma)$.

Recall that given a point $p\in\bfH^3$ and a pair of distinct boundary points 
$\xi\ne \eta\in \partial\bfH^3$ the Busemann function is given by the limit 
\[
    \Bus_p(\xi,\eta)=\lim_{x\to\xi,\ y\to\eta}
    \left(d_{\bfH^3}(p,x)+d_{\bfH^3}(p,y)-d_{\bfH^3}(x,y)\right).
\]
Triangle inequality implies that $\Bus_p(\xi,\eta)\ge 0$. 
Crucial for our purposes, is the fact that
the strict inequality occurs unless $p$ lies on the geodesic line $(\xi,\eta)$: 
\[
    \Bus_p(\xi,\eta)>0\qquad \Longleftrightarrow\qquad p\not\in(\xi,\eta).
\]
The second fact is a consequence of the topological transitivity of the geodesic flow
on the unit tangent bundle to the surface. It can be used to show that for 
any $\xi\ne\eta$ in $\partial\Gamma$ there exists
an infinite sequence $\{\gamma_n\}$ in $\Gamma$ so that 
\[
    \xi=\lim_{n\to\infty}\gamma_{n}^+,\qquad \eta=\lim_{n\to\infty}\gamma_{n}^-
\]
where $\gamma_{n}^-,\gamma_n^+\in\partial\Gamma$ denote the repelling 
and the attracting points of $\gamma_n\in\Gamma$.

\medskip

With these observations we can proceed to the proof of Theorem~\ref{T:main}. 
Using Proposition~\ref{P:cyclic-order}, 
let us pick $(\xi_1, \xi_4)$ and $(\xi_2,\xi_3)$ that are unlinked and 
aligned in $\partial\Gamma$ while $(\phi(\xi_1), \phi(\xi_4))$ 
and $(\phi(\xi_2),\phi(\xi_3))$ are linked in a copy
of a hyperbolic plane $\partial \bfH^2$ contained in the hyperbolic space $\bfH^3$.
Let $p\in\bfH^3$ denote the intersection of the linked geodesic lines
$(\phi(\xi_1), \phi(\xi_4))$ and  $(\phi(\xi_2),\phi(\xi_3))$.
Since these two geodesic lines are distinct, $p\not\in (\phi(\xi_2),\phi(\xi_4))$,
and therefore, using the first fact, we obtain
\[
    \delta=\Bus_p(\phi(\xi_2),\phi(\xi_4))>0.
\]
We can now use the second fact, and find sequences $\{a_n\}$ and $\{b_n\}$ in
$\Gamma$, so that 
\[
    a_n^-\overto{}\xi_1, \qquad a_n^+\overto{}\xi_4,
    \qquad b_n^-\overto{} \xi_2,\qquad b_n^+\overto{} \xi_3.
\]
Denote $A_n=\pi(a_n)$ and $B_n=\pi(b_n)$ the corresponding elements in $\PSL_2(\bbC)$.
Note that $\phi(a_n^\pm)=A_n^\pm$ and $\phi(b_n^\pm)=B_n^\pm$ 
are the repelling/attracting points in $\partial\bfH^3$.
Upon replacing $a_n, b_n$ by their powers, we may assume that 
\[
    \ell_{\bfH^3}(A_n)\overto{}\infty,\qquad \ell_{\bfH^3}(B_n)\overto{}\infty.
\]
Let $p_n^A$ denote the projection of point $p$ to the geodesic line 
$(\phi(a_n^-),\phi(a_n^+))$ which is the axis $\Ax{A_n}$ in $\bfH^3$.
Since $\phi:\partial\Gamma\overto{}\partial\bfH^3$ is continuous,
$A_n^-=\phi(a_n^-)\overto{} \phi(\xi_1)$ and $A_n^+=\phi(a_n^+)\overto{} \phi(\xi_4)$.
This implies  
\[
    d_{\bfH^3}(p_n^A,p)\overto{} 0.
\]
Similarly, denoting by $p_n^B\in\bfH^3$ the projection of $p$ to the geodesic line 
$(\phi(\xi_2),\phi(\xi_3))$ which is the axis $\Ax{B_n}$ in $\bfH^3$, 
we get $d_{\bfH^3}(p_n^B,p)\overto{} 0$. 
\begin{figure}[h]
\centering
\begin{tikzpicture}[scale=4]
\draw[dashed, name path = B] (-120:1cm) node[below right]{$B_n^-$}  .. controls(0,.1) .. (45:1.1cm)node[above right]{$B_n^+$};
\draw[dashed, name path = A] (-1,.5) node[below left]{$A_n^+$} .. controls(0,.3) .. (1,.5)node[below right]{$A_n^-$};
\draw[name path = X1] (-1,.25) node[below]{$\phi(\xi_4)$} -- (1,.25) node[below]{$\phi(\xi_1)$};
\draw[name path = X2] (-140:1cm) node[below]{$\phi(\xi_2)$} -- (60:1cm) node[above]{$\phi(\xi_3)$};

\path[name path = PB] (-1,.1) -- (1,.15);
\path[name path = r] (-1,-.6) -- (1,-.6);
\path[name path = PA] (0,1) -- (-0.1,-1);
\path[name path = APA] (-.6,1) -- (-.6,-1);
\path[name path = xn] (-1,-.45) -- (1,-.45);
\path[name path = yn] (-.6,-1) -- (-.6,1);

\path [name intersections={of = X1 and X2}];
\filldraw (intersection-1) circle(.3pt) node[below right]{$p$};
\path [name intersections={of = B and PB}];
\filldraw (intersection-1) circle(.3pt) node[below right]{$p^B_n$};
\path [name intersections={of = B and r}];
\filldraw (intersection-1) circle(.3pt) node[below right]{$x_n=B_n^{-1}.p^B_n$};
\path [name intersections={of = A and PA}];
\filldraw (intersection-1) circle(.3pt) node[above right]{$p_n^A$};
\path [name intersections={of = A and APA}];
\filldraw (intersection-1) circle(.3pt) node[above]{$y_n=A_n.p^A_n$};
\end{tikzpicture}
\end{figure}

Now consider the points $x_n=B_n^{-1}.p_n^B$ and $y_n=A_n.p_n^A$.
Since $p_n^b=B^n. x_n$ and $x_n$ are on the axis $\Ax{B_n}$ of $B_n$ we have  $d_{\bfH^3}(p_n^b,x_n)=\ell_{\bfH^3}(B_n)$ and 
\begin{equation}\label{e:ellB}
      |d_{\bfH^3}(p,x_n)-\ell_{\bfH^3}(B_n)|\le d_{\bfH^3}(p,p_n^B)\overto{} 0.
\end{equation}
Similarly, 
\begin{equation}\label{e:ellA}
    |d_{\bfH^3}(p,y_n)-\ell_{\bfH^3}(A_n)|\le d_{\bfH^3}(p,p_n^A)\overto{} 0.
\end{equation}
Hence
\[
    \lim_{n\to\infty}x_n=\lim_{n\to\infty} A_n^+=\phi(\xi_2),\qquad
    \lim_{n\to\infty}y_n=\lim_{n\to\infty} B_n^-=\phi(\xi_4).
\]
Therefore 
\begin{equation}\label{e:Busemann}
    \lim_{n\to\infty}\left(d_{\bfH^3}(x_n,p)+d_{\bfH^3}(p,y_n)
        -d_{\bfH^3}(x_n,y_n)\right)
        =\Bus_p(\phi(\xi_2),\phi(\xi_4))=\delta>0.
\end{equation}
We also have 
\[
    \begin{split}
        d_{\bfH^3}\left((A_nB_n).x_n,y_n\right)&= d_{\bfH^3}(A_n.p_n^B,y_n)\\
        &= d_{\bfH^3}(A_n.p_n^B,A_n.p_n^A)=d_{\bfH^3}(p_n^B,p_n^A)\\
        &\le d_{\bfH^3}(p_n^B,p)+d_{\bfH^3}(p,p_n^A)\overto{} 0.
    \end{split}
\]
Using (\ref{e:ellB}), (\ref{e:ellA}), (\ref{e:Busemann}) we deduce
\[
    \lim_{n\to\infty} \left(\ell_{\bfH^3}(A_n)+\ell_{\bfH^3}(B_n)
    -d_{\bfH^3}(A_nB_n.x_n,x_n)\right)=\delta.
\]
Since $\ell_{\bfH^3}(A_nB_n)\le d_{\bfH^3}(A_nB_n.x_n,x_n)$,
it follows that
\[
    \liminf_{n\to\infty}\left(\ell_{\bfH^3}(A_n)+\ell_{\bfH^3}(B_n)
    -\ell_{\bfH^3}(A_nB_n)\right)\ge\delta.
\]
The latter fact can be rewritten as
\[
    \liminf_{n\to\infty}\left(\ell_q(a_n)+\ell_q(b_n)
    -\ell_q(a_nb_n)\right)\ge\delta.
\]
Recall that $(\xi_1,\xi_4)$ and $(\xi_2,\xi_3)$ are unlinked 
and aligned in $\partial\Gamma$, and are approximated by 
$(a_n^-,a_n^+)$ and $(b_n^-,b_n^+)$ respectively. 
Thus, we can find $k\in\bbN$ large enough, so that 
the pair of elements $a=a_k$, $b=b_k$ satisfy
\[
    \ell_{q}(a)+\ell_{q}(b)-\ell_{q}(ab)> \frac{1}{2}\delta
\]
while $(a^-,a^+)$ and $(b^-,b^+)$ are unlinked and aligned.
By Theorem~\ref{T:NC-ineq} the latter condition implies that for every $g\in\Rmn(\Sigma)$ we have
\[
    \ell_g(a)+\ell_g(b)-\ell_g(ab)<0.
\]
Thus
\[
    \frac{\ell_q(a)+\ell_q(b)}{\ell_q(ab)}:\frac{\ell_g(a)+\ell_g(b)}{\ell_g(ab)}
    > \frac{\ell_q(a)+\ell_q(b)}{\ell_q(ab)}
    > 1+\frac{\delta}{2\ell_q(ab)}.
\]
We deduce that for every $g\in\Rmn(\Sigma)$ we get
\[
    \rho_{\Lip}(\delta_q,\delta_g)> \log(1+\frac{\delta}{2\ell_q(ab)})>0.
\]
This completes the proof of Theorem~\ref{T:main}.

\end{document}